\documentclass[10pt]{amsart}
\usepackage[hmargin=3cm,vmargin=3cm]{geometry}

\usepackage[latin1]{inputenc}
\usepackage{xspace,amssymb,amsfonts,euscript, enumerate}
\usepackage{amsthm,amsmath,amscd,stmaryrd,latexsym}
\usepackage{palatino, euscript}

\usepackage{graphicx}
\usepackage[all,dvips]{xy}
\xyoption{line}
\xyoption{arrow}
\xyoption{color}
\SelectTips{cm}{}

\usepackage{tikz }
\usetikzlibrary{matrix, arrows, calc}

\usepackage[dvips]{epsfig}
\usepackage{pinlabel}
\usepackage{psfrag}

\newcommand{\ig}[2]{\vcenter{\xy (0,0)*{\includegraphics[scale=#1]{#2}} \endxy}}

\IfFileExists{comments.sty}{\usepackage{comments}}

\RequirePackage{color}
\definecolor{myred}{rgb}{0.75,0,0}
\definecolor{mygreen}{rgb}{0,0.5,0}
\definecolor{myblue}{rgb}{0,0,0.65}

\RequirePackage{ifpdf}
\ifpdf
  \IfFileExists{pdfsync.sty}{\RequirePackage{pdfsync}}{}
  \RequirePackage[pdftex,
   colorlinks = true,
   urlcolor = myblue, 
   citecolor = mygreen, 
   linkcolor = myred, 
   pagebackref,
   bookmarksopen=true]{hyperref}
\else
  \RequirePackage[hypertex]{hyperref}
\fi

\newtheorem{thm}{Theorem}[section]
\newtheorem{lemma}[thm]{Lemma}

\newtheorem{prop}[thm]{Proposition}

\newtheorem*{prop*}{Proposition}

\theoremstyle{definition}
\newtheorem{defn}[thm]{Definition}

\newtheorem{ex}[thm]{Example}
\newtheorem{example}[thm]{Example}

\theoremstyle{remark}
\newtheorem{remark}[thm]{Remark}

\newtheorem{question}[thm]{Question}

\numberwithin{equation}{section}

  \def\hg{{\mathfrak h}}

    \def\NM{{\mathbb{N}}}

    \def\ZM{{\mathbb{Z}}}

    \def\CC{{\mathcal{C}}}
    \def\DC{{\mathcal{D}}}
  \def\eb{{\mathbf e}}  
  \def\fb{{\mathbf f}}  
    
\def\HB{{\mathbf H}}

    \def\OC{{\mathcal{O}}}

\def\a{\alpha}

\def\l{\lambda}
\def\L{\Lambda}

\let\phi=\varphi

\usepackage{bbm}

\def\Z{{\mathbbm Z}}

\def\1{\mathbbm{1}}

\newcommand{\un}{\underline}

\newcommand{\mf}[1]{\mathfrak{#1}}

\newcommand{\ot}{\otimes}
\newcommand{\pa}{\partial}
\newcommand{\co}{\colon}

\renewcommand{\to}{\rightarrow}

\newcommand\simto{\xrightarrow{
   \,\smash{\raisebox{-0.65ex}{\ensuremath{\scriptstyle\sim}}}\,}}

\newcommand{\define}{\stackrel{\mbox{\scriptsize{def}}}{=}}

\newcommand{\refequal}[1]{\xy {\ar@{=}^{#1}
(-1,0)*{};(1,0)*{}};
\endxy}

\newcommand{\maps}{\colon}

\newcommand{\Hom}{{\rm Hom}}

\newcommand{\End}{{\rm End}}
\newcommand{\END}{{\rm END}}

\newcommand{\Ob}{\textrm{Ob}}

\newcommand{\Ucat}{\mathcal{U}}
\newcommand{\Us}{\mathcal{U}^{\ast}}

\newcommand{\Tr}{{\rm Tr}}
\newcommand{\TLC}{\mathcal{TL}}

\newcommand{\LL}{LL}
\newcommand{\LLL}{\mathbb{LL}}

\newcommand{\lbub}{\xybox{
  (-3,0)*{};(3,0)*{} **\crv{(-3,4) & (3,4)} ?(.0)*\dir{>};
  (3,0)*{};(-3,0)*{} **\crv{(3,-4) & (-3,-4)} ?(.1)*{*};
  (-5,-5)*{}; (5,5)*{};
}}
\newcommand{\rbub}{\xybox{
  (-3,0)*{};(3,0)*{} **\crv{(-3,4) & (3,4)} ?(.0)*\dir{<};
  (3,0)*{};(-3,0)*{} **\crv{(3,-4) & (-3,-4)} ?(.1)*{*};
  (-5,-5)*{}; (5,5)*{};
}}
\begin{document}

\begin{abstract} We compute the trace decategorification of the Hecke category for an arbitrary Coxeter group.  More generally, we introduce the notion of a strictly object-adapted cellular category and calculate the trace for such categories.
\end{abstract}

\title{Trace decategorification of the Hecke category}

\author{Ben Elias} \address{University of Oregon, Eugene}

\author{Aaron D. Lauda} \address{University of Southern California, Los Angeles}

\maketitle

\newcommand{\bigb}[1]{
\begin{tikzpicture}
    \node[draw, thick, fill=blue!20,rounded corners=4pt,inner sep=3pt] () at (0,.5) {\small$#1$};
\end{tikzpicture}}
\newcommand{\UA}{{_{\mathcal{A}}\dot{{\bf U}}}}
\newcommand{\onen}{{\mathbf 1}_{n}}
\newcommand{\UcatD}{\dot{\Ucat}}
\def\nn{\notag}


\section{Introduction}
\label{sec-intro}

The geometrization of representation theoretic objects has proven to be a powerful tool in representation theory, revealing deep insights into the structure of these algebras and their representations.  Geometrization leads to important consequences like integrality and the existence of canonical bases.
A modern viewpoint is that geometrization is a precursor to categorification.  Rather than studying the homology or $K$-theory of a variety $X$,  one can alternatively consider the Grothendieck group of an appropriate category of sheaves on $X$.  In this context, algebra generators are lifted to functors, and relations between generators hold up to natural isomorphisms of functors.  One of the goals of categorification is to identify the algebraic objects that govern the natural transformations between such functors.

One of the advantages categorification and its purely algebraic formulations provide is that one can often surpass the geometric realm and vastly extend various results.  For example, the categorification of positive parts of quantum groups $U^+_q(\mathfrak{g})$ associated non-symmetric Cartan data give rise to bases for the $U^+_q(\mathfrak{g})$ which are necessarily positive~\cite{KL2}; the canonical bases associated to non-symmetric Cartan data are known to violate positivity, so in this instance, categorification gives rise to new results not accessible from geometry.

For the purposes of this paper, we will focus on another example where algebraic categorifications significantly extend geometric results. The Hecke algebra associated to the Weyl group of
a reductive algebraic group was geometrically realized in the seminal work of Kazhdan and Lusztig~\cite{KaL1}. Soergel introduced an algebraic reinterpretation of this geometric
construction that allowed for the categorification of the Hecke algebra associated to an arbitrary Coxeter system~\cite{Soe3}. He constructed a family of graded bimodules over a polynomial
ring, which form an additive category now called the category of Soergel bimodules, and proved that the split Grothendieck group of this category is the Hecke algebra. For Weyl groups, he
provided a fully ext-faithful functor (i.e. sending the graded extension space isomorphically to the graded hom space) from semisimple perverse sheaves (in the category considered by
Kazhdan and Lusztig) to Soergel bimodules. A further generalization of Soergel bimodules is the diagrammatic category $\DC$ constructed by the first author and Williamson~\cite{EWGR4SB}.


In the geometric realm, one can apply a range of ``decategorification" functors, ranging from various flavors of (equivariant) (co)homology and $K$-theory to arbitrary generalized
cohomology theories. In many instances, relations between algebraic objects can be realized geometrically by relating these decategorification functors. For example, the classical Chern
character map relates $K_0$ to a completion of $H^{\ast}(X)$. Applied to various geometrizations, the Chern character relates the loop algebra $U_q(L\mf{g})$ with Yangians
$Y_{\hbar}(\mf{g})$, or the affine Hecke algebra with the graded Hecke algebra. In each case, the algebraic objects arising from $K$-theory and (co)homology are related  via the (geometric) Chern character.

Recently, it has become clear that categorification provides an algebraic analog of the Chern character.
Given an additive category $\CC$ one can consider its split Grothendieck group $K_0(\CC)$.  Alternatively, one can consider its \emph{trace decategorification} $\Tr(\CC)$ given by the quotient of $\End(\CC) = \oplus_{X \in \Ob(\CC)} \End(X)$ by the \emph{trace relation}
\begin{equation} \label{tracerelation}
fg - gf = 0
\end{equation}
for any $f \in \Hom(X,Y)$ and $g \in \Hom(Y,X)$ with $X,Y \in \Ob(\CC)$.  The trace is sometimes called the cocenter or 0th-Hochschild homology.  There is always a \emph{Chern character} map
\[
 h_{\CC} \maps K_0(\CC) \to  \Tr(\CC)
\]
sending the usual Grothendieck group $K_0(\CC)$ of $\CC$ (or of the Karoubi envelope $\dot{\CC}$, if $\CC$ is not already Karoubian) to $\Tr(\CC)$; it sends the class of an object $[X]$ to the image of the identity map $1_X$. When $\CC$ is monoidal, both $K_0(\CC)$ and $\Tr(\CC)$ have the structure of a ring.

When $\CC$ is graded, there are two natural things one can do:
\begin{itemize}
 \item Consider $\Tr(\CC)$ as stated, where $\End(X)$ refers to endomorphisms of degree $0$. The resulting trace decategorification will naturally be a $\ZM[q,q^{-1}]$-module, consistent with the usual $\ZM[q,q^{-1}]$-module structure on the Grothendieck group.
 \item Consider the category $\CC^*$ whose endomorphism rings are the graded vector spaces $\END(X)$ given by endomorphisms of all degrees. In $\CC^*$ all grading shifts of an object are isomorphic, so that the Chern character map $[\CC] \to \Tr(\CC^*)$ will factor through $q=1$. Despite this, $\Tr(\CC^*)$ tends to be
richer than $\Tr(\CC)$, because it detects endomorphisms of nonzero degree.
\end{itemize}

The trace decategorification has been computed in the context of several well-studied categorifications~\cite{CLLS, BHLW, BHLZ,Marco,BGHL,SVV}. An elementary introduction to the to the trace decategorification functor can be found in \cite{BGHL}. The papers \cite{BHLZ,BHLW,SVV} treat the case of the categorified quantum group $\Ucat_Q(\mf{g})$ and its representations.  It was previously established that $K_0(\Ucat_Q^{\ast}(\mf{g}))$ is isomorphic to  Lusztig's integral idempotent form of quantum enveloping algebra of $\mf{g}$.  In \cite{BHLZ,BHLW,SVV} the trace $\Tr(\Ucat_Q^{\ast}(\mf{g}))$ is  identified with the current algebra $U(\mf{g}[t])$ associated to $\mf{g}$. In a closely related story, the traces of various Heisenberg categories are identified with $W$-algebras~\cite{CLLS,SVV}.

In this paper we study the trace decategorification of the Hecke category. Here, the \emph{Hecke category} will refer to the graded, additive, monoidal category $\DC = \DC(W, S,\hg)$,
described diagrammatically by generators and relations by Elias-Williamson in \cite{EWGR4SB}. This diagrammatic presentation was given earlier in special cases by Elias-Khovanov \cite{EKh},
Elias \cite{EDihedral}, and Libedinsky \cite{LibRA}. There is one such category $\DC$ for each Coxeter system $(W,S)$ and \emph{realization} $\hg$, a realization being roughly a reflection
representation equipped with a choice of simple roots and coroots. It is discussed at length in \cite{EWGR4SB} how the category $\DC$ agrees with all other models of the Hecke category
(e.g. equivariant perverse sheaves or parity sheaves on the flag variety, Soergel bimodules, etc) for realizations where these other models are well-behaved, and has the correct behavior
even when these other models break down. For a detailed introduction to the Hecke category and its long history, please see the introduction to \cite{EWGR4SB}.

However, $\DC$ has another structure which allows for a simple computation of its trace decategorification: it is a strictly object-adapted cellular category, or SOACC. Libedinsky
\cite{LibLL} defined a basis of morphisms, known as \emph{Libedinsky's light leaves}; in \cite{EWGR4SB} it was shown that light leaves form a cellular basis for $\DC$. In fact, though the
terminology was not used there, it was shown in \cite{EWGR4SB} that $\DC$ is an SOACC, which is roughly a cellular category where the cellular basis
factors in a compatible way as compositions of two ``trapezoidal" maps (see \cite[\S 1.5]{EWGR4SB}). Using this factorization and \eqref{tracerelation}, one can show that identity maps span
the trace decategorification, and that the Chern character map is an isomorphism after base change. The proof of this result fits into a much more general framework. Section
\ref{sec-cellcat} discusses strictly object-adapted cellular categories and their trace decategorifications. This result immediately computes the trace of $\DC$ as a vector space, though it
does not compute the ring structure.

One has an isomorphism $[\dot{\DC}] \cong \HB$ from the Grothendieck group of the Karoubi envelope to the Hecke algebra of $W$. The Chern character map gives an isomorphism
of rings $$\HB \ot_\ZM \Bbbk \simto \Tr(\DC),$$ where $\Bbbk$ is the (commutative complete local) ring over which $\hg$ is defined. For $\DC^*$, the Chern character gives an isomorphism
$$\ZM[W] \ot_\ZM R \simto \Tr(\DC^*),$$ where $R = \Bbbk[\hg] = \END(\1)$ is the polynomial ring over $\hg$, and the endomorphism ring of the monoidal identity. However, this is not an
isomorphism of rings. We compute the ring structure in Section \ref{sec-traceofD} and show that $$\ZM[W] \ltimes R \cong \Tr(\DC^*).$$ These results are given in Theorem \ref{mainthm1} and
Theorem \ref{mainthm2}.

The arguments used in the computation of the trace for a SOACC immediately generalize to a more general class of categories we call \emph{fibred} strictly object-adapted cellular categories.   These categories are an appropriate abstraction in which the methods developed in \cite{BHLZ} can be used to explicitly compute the trace.


There is an action of the trace of a monoidal category on the center of that category; in Theorem \ref{mainthm3}, we show for $\DC$ that this is the sign representation of $\HB =
[\dot{\DC}]$ induced to either $\Tr(\DC)$ or $\Tr(\DC^*)$. There is an action of the trace of a monoidal category on the trace of any categorical representation. For $\DC$, a broad class of
interesting categorical representations can be obtained using cellular subquotients. However, these subquotients are not well understood in general, so we do not compute their traces here.
It should be interesting to compute the trace decategorification in known cases, such as the cell modules constructed in \cite{ETemperley} for two-row partitions in type $A$.

%
%
%

\subsection*{Acknowledgements}
A.D.L. was partially supported by NSF grant DMS-1255334 and by the John Templeton Foundation.  Both authors are grateful to Anna Beliakova and Kazuo Habiro for helpful discussions on trace categorifications.

\section{Strongly object-adapted cellular categories}
\label{sec-cellcat}

Cellular categories were defined by Westbury in \cite{Westbury} as a generalization of cellular algebras. Inside a cellular category, the morphism space $\Hom(X,Y)$ between any two objects
has a fixed basis $$\{c^\l_{S,T}\} \textrm{ for } \l \in \L, S \in M(\l,Y), T \in E(X,\l)$$ parametrized by a ``cell" $\l$ and two sets $E(X,\l)$ and $M(\l,Y)$. Many of the cellular
categories in nature have the property that $c^\l_{S,T} = c_S \circ c_T$ factors as a composition, passing through an object associated to the cell $\l$; in this context, the sets $E(X,\l)$
and $M(\l,Y)$ can be viewed as subsets of $\Hom(X,\l)$ and $\Hom(\l,Y)$. The Temperley-Lieb category is an exemplar of this phenomenon, as are many similar diagram categories. The category
$\DC$ is a less well-known example.

We encapsulate this common behavior in the notion of a strictly object-adapted cellular category. This definition has appeared previously in unpublished work of the first author, though
surprisingly we have not found it in the known literature.

\subsection{Object-adapted cellular categories}
\label{subsec-defn-SOACC}

Let us recall Westbury's definition of a cellular category.

\begin{defn} A \emph{cellular category} over a commutative base ring $R$ is the following data:
\begin{itemize} \item An $R$-linear category $\CC$.
\item A set $\L$ with a partial order $\le$. We assume that $\le$ has the descending chain condition.
\item For each $X \in \Ob(\CC)$ and each $\l \in \L$, finite sets $M(\l,X)$ and $E(X,\l)$. These two sets are in a
fixed bijection, and we call the map $M(\l,X) \leftrightarrow E(X,\l)$ in either direction by the name $\iota$.
\item A map $c^\l \co M(\l,Y) \times E(X,\l) \to \Hom(X,Y)$. We write $c_{S,T}$ or $c^\l_{S,T}$ instead of $c^\l(S,T)$.
\end{itemize}
We let $\CC_{\le \l}$ denote the $R$-linear subcategory whose morphisms are spanned by those $c^\mu_{S,T}$ with $\mu \le \l$. The conditions below will imply that $\CC_{\le \l}$ is actually an ideal, i.e. it is closed under composition with arbitrary morphisms. We let $\CC_{< \l}$ be defined similarly.

This data satisfies the following conditions: \begin{enumerate}
\item For fixed $X,Y \in \Ob(\CC)$, the set $\{c^\l_{S,T}\}$ forms an $R$-basis for $\Hom(X,Y)$.
\item There is a (unique) $R$-linear
contravariant involution of $\CC$, also denoted $\iota$, which sends $c^\l_{S,T} \mapsto c^\l_{\iota(T), \iota(S)}$.
\item There is a function $\ell \co \Hom(Y,Z) \times M(\l,Y) \times M(\l,Z) \to R$ such that, for every $f \in \Hom(Y,Z)$,
$S \in M(\l,Y)$, one has \begin{equation} \label{oldcellularcondition} f \circ c^\l_{S,T} = \sum_{S' \in M(\l,Z)} \ell(f,S,S') c^\l_{S',T} \textrm{ modulo } \CC_{< \l}. \end{equation} This equation holds for any $T \in E(X,\l)$ for any $X$, and is independent of the choice of $T$ and $X$.
\end{enumerate}

A \emph{cellular algebra} is a cellular category with one object.\end{defn}

\begin{remark} In the literature, the sets $M(\l,X)$ and $E(X,\l)$ are usually identified and go under the name $M(X,\l)$. We separate them for reasons which are soon to become obvious.
\end{remark}

\begin{lemma} The subcategories $\CC_{\le \l}$ and $\CC_{< \l}$ are two-sided ideals. \end{lemma}

\begin{proof} This is clear from \eqref{oldcellularcondition}. Note that, applying $\iota$ to \eqref{oldcellularcondition}, one obtains a similar formula for right multiplication by $f$.
\end{proof}

The notion of a strictly object-adapted cellular category will be a rigidification of that of a cellular category.

\begin{defn} A \emph{strictly object-adapted cellular category} or \emph{SOACC} over a commutative base ring $R$ is the following data:
\begin{itemize} \item An $R$-linear category $\CC$.
\item A set of objects $\L \subset \Ob(\CC)$, and a partial order $\le$ on $\L$. We assume that $\le$ has the descending chain condition.
\item For each $X \in \Ob(\CC)$ and each $\l \in \L$, finite sets $M(\l,X)$ and $E(X,\l)$. These two sets are in a
fixed bijection, and we call the map $M(\l,X) \leftrightarrow E(X,\l)$ in either direction by the name $\iota$.
\item A map $c \co M(\l,X) \to \Hom(\l,X)$ and a map $c \co E(X,\l) \to
\Hom(X,\l)$. We write $c_S$ instead of $c(S)$ (resp. $c_T$ instead of $c(T)$), for $S \in M(\l,X)$ (resp. $T \in E(X,\l)$).
\end{itemize}
For $T \in E(X,\l)$ and $S \in M(\l,Y)$ we let
\begin{equation} \label{cellbasisdef} c^\l_{S,T} \define c_S \circ c_T. \end{equation}
We let $\CC_{\le \l}$, $\CC_{< \l}$ be defined as above.

This data satisfies the following conditions: \begin{enumerate}
\item For fixed $X,Y \in \Ob(\CC)$, the set $\{c^\l_{S,T}\}$ forms an $R$-basis for $\Hom(X,Y)$.
\item There is an $R$-linear
contravariant involution of $\CC$, also denoted $\iota$, for which $\iota(c_S) = c_{\iota(S)}$. This property defines $\iota$ uniquely, because it implies that $\iota \co \Hom(X,Y) \to
\Hom(Y,X)$ sends $c^\l_{S,T} \mapsto c^\l_{\iota(T), \iota(S)}$.
\item There is a function $\ell \co \Hom(Y,Z) \times M(\l,Y) \times M(\l,Z) \to R$ such that, for every $f \in \Hom(Y,Z)$ and
$S \in M(\l,Y)$, one has \begin{equation} \label{cellularcondition} f \circ c_S = \sum_{S' \in M(\l,Z)} \ell(f,S,S') c_S' \textrm{ modulo } \CC_{< \l}. \end{equation}
\item The sets $M(\l,\l) \cong E(\l,\l)$ consist of a single point $*$, and both maps $c_*$ are equal to the identity map $\1_\l \in \End(\l)$.
\end{enumerate} \end{defn}

\begin{defn} A \emph{graded SOACC} is an SOACC for which $R$ is graded, $\CC$ is graded, and the maps $c_S$ are homogeneous. For example, there is a decomposition $$M(\l,X) = \coprod_{k \in
\ZM} M^k(\l,X),$$ where $M^k(\l,X)$ indexes those cellular maps $c_S$ which have degree $k$. We write $\Hom^k(X,Y)$ for homogeneous degree $k$ morphisms from $X$ to $Y$. \end{defn}

It is clear that any SOACC is a cellular category, as condition \eqref{oldcellularcondition} is simply \eqref{cellularcondition} precomposed with $c_T$. Therefore $\CC_{< \l}$ is an ideal.

\begin{lemma} For $\l, \mu \in \L$, the sets $M(\l,\mu)$ and $E(\mu, \l)$ are empty unless $\l \le \mu$. \end{lemma}

\begin{proof} Suppose that $S \in M(\l,\mu)$. Then because $\1_\mu$ is inside the ideal $\CC_{\le \mu}$, so is $c_S$. However, $c_S = c_{S,*}$ for $* \in E(\l,\l)$, and therefore $c_S$ is a cellular basis element in cell $\l$. Thus $\l \le \mu$. \end{proof}

\begin{lemma} \label{endoSOACC} Modulo $\CC_{< \l}$, the endomorphism ring $\End(\l)$ is equal to $R \cdot \1_\l$. \end{lemma}

\begin{proof} Modulo lower terms, only $c^\l_{*,*}$ survives. \end{proof}

\begin{lemma} \label{equivcondition} Given the data of a potential SOACC $\CC$ such that conditions (1), (2), and (4) hold, the remaining SOACC condition (3) is equivalent to the fact that $\CC_{\le \l}$ is an ideal for all $\l$. \end{lemma}

\begin{proof} It is clear that $\{c_{S'}\}_{S' \in M(\l,Z)}$ forms a basis for $\Hom(\l, Z)$ modulo $\CC_{< \l}$. From this, \eqref{cellularcondition} is clear. \end{proof}

Lemma \ref{equivcondition} has no analog for ordinary cellular algebras; the cellular condition \eqref{oldcellularcondition} becomes more natural after the cellular basis is factored.

\begin{defn} \label{pairingSOACC} The \emph{cellular pairing} $\phi^\l_X \co E(X, \l) \times M(\l,X) \to R$ sends $(T,S) \mapsto c_T \circ c_S$, viewed as an element of $\End(\l)$ modulo $\CC_{< \l}$. That is,
the pairing picks out the coefficient of the identity map inside $c_T \circ c_S$. \end{defn}

\begin{lemma} \label{pairmultSOACC} Modulo $\CC_{< \l}$, one has $c^\l_{S,T} \circ c^\l_{U,V} = \phi^\l(T,U) c^\l_{S,V}$. \end{lemma}

\begin{proof} This is clear, by writing the composition as $c_S(c_T c_U)c_V$. \end{proof}

Let us compare cellular categories with SOACCs.

In a cellular category, there is no requirement that $c^\l_{S,T}$ factors, and thus no identification of $\L$ with a subset of $\Ob(\CC)$. For example, any cellular algebra can be thought of
as a cellular category with one object, and there are no other objects for the cellular basis to factor through; if this is an SOACC, then the cellular algebra is just the base ring $R$.
However, many examples of cellular algebras can be evidenced as endomorphism rings in SOACCs (see \S\ref{subsec-egs} for examples). Thus it is interesting to ask which cellular
categories can be enlarged to become SOACCs.

\begin{question} Is there a natural algorithm which takes a cellular category $\CC$ and produces an SOACC $\hat{\CC}$? We require that $\Ob(\hat{\CC}) = \Ob(\CC) \coprod \L$, and that the
objects originally in $\CC$ come from a fully faithful functor $\CC \to \hat{\CC}$, which preserves the cellular basis. \end{question}

At the moment there are no known cellular categories which are proven not to be embeddable in an SOACC. It also appears in examples that there may be a great deal of freedom in the
construction of $\hat{\CC}$, because one can allow the sets $M(\l,\mu)$ to be rather large. The cellular structure does not seem to encode sufficiently the lower terms which appear when
cellular basis elements are composed; whether there is a natural choice remains to be seen.

Note that embedding a cellular category in an SOACC may change its decategorification. For example, in a cellular category there may exist $\l \in \L$ such that $\phi^\l_X = 0$ for all $X$.
Such a $\l$ will not contribute a simple module to the representation theory of $\CC$. However, this will never happen in an SOACC because $\phi^\l_\l(*,*)=1$. This was desirable because it
gave meaning to the cellular pairing as a local intersection form, which computes the coefficient of the identity in the composition of two maps. The lack of degeneracy in SOACCs allows one
to compute their Grothendieck groups and trace decategorification with ease. This is one of the great advantages of SOACCs over cellular categories.

Being an SOACC is closed under categorical equivalence. The objects in $\L$ are sent to some other objects, which serve as the cellular poset in the target category. However, the cellular
structure itself is clearly not intrinsic, and an autoequivalence may alter the chosen structure. The additive closure of an SOACC is also an SOACC in a natural way. One sets $E(X \oplus Y,
\l) = E(X,\l) \coprod E(Y,\l)$, and constructs the cellular maps and the cellular basis in a matrix-like fashion. However, the Karoubi envelope of an SOACC is typically not an SOACC in a
natural way. Given an idempotent pair $(X,e)$, it is not clear how to define $E((X,e),\l)$, as the idempotent need not play nicely with the cellular basis. This entire paragraph applies
equally to cellular categories.

Thus neither notion is intrinsic, though SOACCs are more rigid than cellular categories. For a cellular category $\CC$, one can obtain another cellular structure by changing the basis
$c^\l_{S,T}$ in an upper-triangular way, adding lower terms to each basis element (consistently with $\iota$). Similarly one can change the morphisms $c_S$ by adding lower terms to obtain a
new SOACC structure. The induced change of basis on $\{c^\l_{S,T}\}$ is more restrictive than was allowed for a cellular category. In either context, we call such an operation a \emph{cellular change of basis}.

\begin{remark} The reader may be wondering what the adjective ``strictly" means for a strictly object-adapted cellular category. One of the most famous cellular algebras in the literature
is the Hecke algebra itself in type $A$. The first author wished to view the Hecke algebra as an endomorphism ring inside a cellular category where the cellular basis would factor, and a
natural choice is the Hecke algebroid (or Schur algebroid) defined by Williamson \cite{WilSSB}. However, this Hecke algebroid is not naturally an SOACC. Instead, one can identify each $\l
\in \L$ with a set of objects (not a single object) in the Hecke algebroid, with fixed transition maps between these objects. Now, the cellular basis factors as $c^\l_{S,T} = c_S \circ \psi
\circ c_T$, where $\psi$ is the transition map between the target of $c_T$ and the source of $c_S$. This results in the much more technical (though more intrinsic) theory of
\emph{object-adapted cellular categories}, which remains largely undeveloped. \end{remark}

\subsection{Examples}
\label{subsec-egs}

\begin{ex} Consider the \emph{Temperley-Lieb category} $\TLC$ defined, for instance, in \cite{WestburyTL}. One can set $\L = \Ob(\TLC) = \NM$, with the usual order $\le$. Then $M(m,n)$ is
the set of cup diagrams from $m$ to $n$, and $E(n,m)$ the set of cap diagrams. Given a cup (resp. cap) diagram $S$, $c_S$ is the corresponding morphism in $\TLC$. The involution $\iota$
flips a diagram upside-down. This equips $\TLC$ with the structure of an SOACC over $R = \ZM[q,q^{-1}]$. \end{ex}

\begin{ex} Let $R = \Bbbk$ be a field, and let $V = \Bbbk^I$ be a finite dimensional vector space. The endomorphism ring $A = \End(V)$, also known as $I \times I$ matrices, is a prototypical
cellular algebra with a single cell. One has $M = E = I$, the cellular basis is the usual matrix entry basis, and $\phi$ is the usual inner product corresponding to the identity matrix. In
similar fashion, one can equip the category $\textrm{Vect}_{\Bbbk}$ with the structure of an SOACC, with a single cell corresponding to the one-dimensional vector space $\Bbbk$. The SOACC
structure is determined by a choice of basis for every vector space (so long as each $\phi$ is the identity matrix), which is clearly not an intrinsic notion.

The Karoubi envelope of $\textrm{Vect}$ is equivalent to $\textrm{Vect}$, and thus can be equipped with an SOACC structure. However, it does not have any natural SOACC structure, as a choice
of basis for a vector space will not descend to a choice of basis for a summand. \end{ex}

\begin{ex} Let $R=\Bbbk$ be a field, let $I$ index a set of variables $\{x_i\}$. Let $A = \Bbbk[I]/J_2$ denote the quotient of the polynomial ring in variables $x_i$ by the ideal generated
by all elements of degree $2$. This ring can be equipped with the structure of a cellular category (with one object, $\star$) as follows. One has $\L = I \coprod \{\star\}$, with $\star >
i$ for all $i \in I$, and no relations between nonequal elements of $I$. For every $\l \in \L$, the sets $E(\star,\l)$ and $M(\l,\star)$ are singletons, with $c^\star = 1$ and $c^i = x_i$.
The antiinvolution $\iota$ is the identity map. It is easy to see that $A$ is a (graded) cellular algebra. Note that $\phi^\star=1$, while $\phi^i=0$ for all $i \in I$.


Now consider the minimal way to embed $A$ as the endomorphism ring (of an object called $\star$) inside an SOACC $\CC$. Add one new object for each $i \in I$, and factor each polynomial
generator $x_i$ as a composition of two arrows $x_i = c_{\overline{i}} c_i$, such that $c_i c_{\overline{i}} = 0$ inside $\End(i)$. We leave the details to the reader. Now one has $\phi^i_X
= 0$, while $\phi^i_i = 1$ as it must for any SOACC.

There is only one nonzero simple module $L_{\star}$ for $A$, on which each $x_i$ acts as zero. However, $\CC$ also has a nonzero simple module $L_i$ for each $i \in I$, on which the identity
morphism $\1_i$ acts by the identity, and the identity morphism $\1_j$ acts as zero for each $j \ne i$. Thus the Grothendieck group of $\CC$ is larger than that of $A$.

Similarly, the ring $A$ is commutative, so $\Tr(A) = A$, and $x_i \ne 0$ in $\Tr(A)$. Meanwhile, $\Tr(\CC)$ is spanned by identity maps, and $x_i = 0$ in $\Tr(\CC)$. \end{ex}

\begin{ex} The main example of our paper is the Hecke category $\DC(W,S,\hg)$. We summarize the SOACC structure here; for details on the category itself, see \cite{EWGR4SB}. The category
$\DC$ is linear over the polynomial ring $R = \OC(\hg)$. (In fact, morphisms are $R$-bimodules, but we arbitrarily choose the right action for our $R$-linearity.) An object in $\DC$ is a
sequence $\un{w}$ of simple reflections in $S$. The poset $\L$ is isomorphic to $W$ with its Bruhat order. We identify $\L$ with a set of objects in $\DC$ by choosing an arbitrary reduced
expression for each element of $W$.

Any sequence $\un{w}$ of length $d$ has $2^d$ subsequences, which can be encoded in sequences $\eb$ of the exponents $0$ and $1$. We write $\eb \subset \un{w}$, and write $\un{w}^{\eb}$ for
the element expressed by the subsequence. We set $$M(x, \un{w}) = E(\un{w}, x) = \{\eb \subset \un{w} \; | \; \un{w}^\eb = x\}.$$ In \cite[\S 6]{EWGR4SB}, an algorithm is given (following
Libedinsky \cite{LibLL}) which produces, for any $\eb \in E(\un{w},x)$, a morphism $\LL_{\un{w}, \eb}$ in $\DC$ from $\un{w}$ to the chosen reduced expression $\un{x}$. This so-called
\emph{light leaves map} will be $c^x_{\eb}$. The map $\iota$ which flips diagrams in $\DC$ upside-down is an antiinvolution. Flipping $\LL_{\un{w},\eb}$ upside-down yields the corresponding
morphism for $\eb \in M(x,\un{w})$, which is denoted $\overline{\LL}_{\un{w},\eb}$. This fixes the data of an SOACC. We write $\LLL^x_{\eb, \fb} = \overline{\LL}_{\un{w},\eb} \circ
\LL_{\un{y},\fb}$ for the composition of two light leaves for subexpressions which express the same element $x \in W$; these elements of the purported cellular basis are \emph{double leaves
maps}.

The algorithm for constructing a light leaves map is not deterministic and depends on some choices. Making different choices will correspond to a cellular change of basis. When $\un{x}$ is
the chosen reduced expression for $x \in W$, there is a unique subexpression $\eb$ with $\un{x}^\eb = x$, namely the ``all ones" subexpression. The algorithm allows for $\LL_{\un{x}, \eb}$
to be the identity map $\1_{\un{x}}$ in this case, and we enforce that here (though this choice was not required in \cite{EWGR4SB}).

It was proven in \cite{EWGR4SB} that the double leaves forms a basis for morphism spaces in $\DC$ under the (right) action of $R$. Properties (1) and (2) of an SOACC are now obvious.
Property (4) follows by the previous paragraph. Finally, it was shown in \cite{EWGR4SB} that $\DC_{\le x}$ is an ideal which, by Lemma \ref{equivcondition}, implies that property (3) of an
SOACC holds. Thus $\DC$ is an SOACC. \end{ex}

\subsection{Fibered object-adapted cellular categories}
\label{subsec-defn-SOAaCC}

We are interested in computing the Grothendieck and trace decategorifications of SOACCs, which we will do in the next section. However, the same techniques apply to a slightly broader class
of categories. In an SOACC the endomorphism ring of an object $\l \in \L$, modulo lower terms, is spanned by the identity map over the base ring $R$. Now we allow it to be a more interesting
commutative ring $A_\l$, depending on $\l$. When we compute the Grothendieck group, we will assume that $A_\l$ is local.

\begin{defn} A \emph{fibred strictly object-adapted cellular category} or \emph{fibred SOACC} over a commutative base ring $R$ is the following data:
\begin{itemize} \item An $R$-linear category $\CC$.
\item A set of objects $\L \subset \Ob(\CC)$, and a partial order $\le$ on $\L$. We assume that $\le$ has the descending chain condition.
\item For each $X \in \Ob(\CC)$ and each $\l \in \L$, finite sets $M(\l,X)$ and $E(X,\l)$. These two sets are in a
fixed bijection, and we call the map $M(\l,X) \leftrightarrow E(X,\l)$ in either direction by the name $\iota$.
\item A map $c \co M(\l,X) \to \Hom(\l,X)$ and a map $c \co E(X,\l) \to
\Hom(X,\l)$. We write $c_S$ instead of $c(S)$ (resp. $c_T$ instead of $c(T)$), for $S \in M(\l,X)$ (resp. $T \in E(X,\l)$).
\item A commutative free $R$-algebra $A_\l \subset \End(\l)$ for each $\l \in \L$, equipped with an $R$-linear antiinvolution $\iota$.
\end{itemize}
For $T \in E(X,\l)$ and $S \in M(\l,Y)$ and $a \in A_\l$ we let
\begin{equation} \label{cellbasisdef2} c^\l_{S,a,T} \define c_S \circ a \circ c_T. \end{equation}
We let $\CC_{\le \l}$, $\CC_{< \l}$ be defined as before.

This data satisfies the following conditions: \begin{enumerate}
\item For fixed $X,Y \in \Ob(\CC)$, the set $\{c^\l_{S,a_i,T}\}$ forms an $R$-basis for $\Hom(X,Y)$, where $a_i$ ranges over an $R$-basis for $A_\l$.
\item There is an $R$-linear
contravariant involution of $\CC$, also denoted $\iota$, which extends the antiinvolution on each $A_\l$, and for which $\iota(c_S) = c_{\iota(S)}$. This property defines $\iota$ uniquely, because it implies that $\iota \co \Hom(X,Y) \to
\Hom(Y,X)$ sends $c^\l_{S,a,T} \mapsto c^\l_{\iota(T),\iota(a), \iota(S)}$.
\item There is a function $\ell \co \Hom(Y,Z) \times M(\l,Y) \times M(\l,Z) \to A_\l$ such that, for every $f \in \Hom(Y,Z)$ and
$S \in M(\l,Y)$, one has \begin{equation} \label{cellularcondition2} f \circ c_S = \sum_{S' \in M(\l,Z)}  c_S' \ell(f,S,S') \textrm{ modulo } \CC_{< \l}. \end{equation}
\item The sets $M(\l,\l) \cong E(\l,\l)$ consist of a single point $*$, and both maps $c_*$ are equal to the identity map $\1_\l \in \End(\l)$.
\end{enumerate} \end{defn}

Thus any morphism can be written as a sum of morphisms $c^\l_{S,a,T}$; no $R$-coefficients are needed, because they can be absorbed into the element $a$.

One must adjust the results dealing with $\End(\l)$ from the SOACC case.

\begin{lemma} (c.f. Lemma \ref{endoSOACC}) Modulo $\CC_{< \l}$, the endomorphism ring $\End(\l)$ is equal to $A_\l$. \end{lemma}
	
\begin{proof} Modulo lower terms, only $c^\l_{*,a,*}$ survives, for various $a \in A_\l$. \end{proof}

\begin{defn} (c.f. Definition \ref{pairingSOACC}) The \emph{cellular pairing} $\phi^\l_X \co E(X, \l) \times M(\l,X) \to A_\l$ sends $(T,S) \mapsto c_T \circ c_S$, viewed as an element of
$\End(\l)$ modulo $\CC_{< \l}$. That is, the pairing picks out the coefficient of the identity map inside $c_T \circ c_S$. \end{defn}

\begin{lemma} \label{pairmultSOAaCC} (c.f. Lemma \ref{pairmultSOACC}) Modulo $\CC_{< \l}$, one has $c^\l_{S,a,T} \circ c^\l_{U,b,V} = c^\l_{S, a \phi^\l(T,U) b,V}$. \end{lemma}

\begin{proof} This is clear. \end{proof}

Any SOACC is clearly a fibred SOACC where each fibre ring $A_\l$ is just $R$. A fibred SOACC is typically not a cellular category, unless the rings $A_\l$ are themselves cellular algebras (and even then,
being a cellular category is an additional structure). Nonetheless, when $A_\l$ behaves nicely, they share many of the same properties as cellular categories. This will be seen in the next
section.

The following example and the computation of its trace decategorification in \cite{BHLZ} motivates the study of fibred SOACCs.
\begin{example}
In \cite{Lau1} a 2-category $\mathcal{U}=\mathcal{U}(\mf{sl}_2)$ was defined whose Karoubi envelope $\UcatD$ has Grothendieck group isomorphic to the integral idempotented form $\UA$ of the quantum enveloping algebra of $\mf{sl}_2$. The indecomposable 1-morphisms of this 2-category correspond bijectively to elements of the Lusztig canonical basis  $\mathbb{B}$ of $\UA$
\begin{enumerate}[(i)]
     \item $E^{(a)}F^{(b)}1_{n} \quad $ for $a,b\geq 0$,
     $n\in\Z$, $n\leq b-a$,
     \item $F^{(b)}E^{(a)}1_{n} \quad$ for $a$,$b\geq 0$, $n\in\Z$,
     $n\geq
     b-a$,
\end{enumerate}
where $E^{(a)}F^{(b)}1_{b-a}=F^{(b)}E^{(a)}1_{b-a}$. Let
 $_m\mathbb{B}_n$ be set of elements in $\mathbb{B}$ belonging to $1_m(\UA)1_n$.

 The defining relations for the algebra $\UA$
all lift to explicit isomorphisms in $\dot{\Ucat}$ (see \cite[Theorem 5.1, Theorem 5.9]{KLMS}) after associating to
 each $x \in \Bbb B$  a $1$-morphism in $\UcatD$ as follows:
\begin{equation} \label{eq_basis}
  x \mapsto \mathcal{E}(x) := \left\{
\begin{array}{cl}
  \mathcal{E}^{(a)}\mathcal{F}^{(b)}\onen & \text{if $x=E^{(a)}F^{(b)}1_n$,} \nn \\
  \mathcal{F}^{(b)}\mathcal{E}^{(a)}\onen & \text{if $x=F^{(b)}E^{(a)}1_n$.} \nn
\end{array}
  \right.
\end{equation}

In \cite[Proposition 5.15]{KLMS} a basis is provided for the hom spaces between these indecomposable 1-morphisms.   For $n\in\Z$, $a,b \geq 0$, $\delta\in\Z$, let us define the following
sets:
\begin{gather*}
 B^+(n,a,b,\delta) := \\
 \{f^{b,a,i,j}_{\lambda,\mu,\nu,\sigma,\tau}\onen |
0\leq i,j\leq \min(a,b), \delta=i-j, \lambda \in P(a-j), \mu \in P(b-j), \nu \in P(i), \sigma \in P(j), \tau \in P \},
\end{gather*}
\begin{gather*}
B^-(n,a,b,\delta) := \\
\{g^{a,b,i,j}_{\lambda,\mu,\nu,\sigma,\tau}\onen |
0\leq i,j\leq \min(a,b), \delta=i-j, \lambda \in P(a-j), \mu \in P(b-j), \nu \in P(i), \sigma \in P(j), \tau \in P\},
\end{gather*}
where the $P(a)$ denotes the set of all partitions with at most $a$
parts, $P$ denote the collection of all partitions of arbitrary size, and
\begin{equation*}
 f^{b,a,i,j}_{\lambda,\mu,\nu,\sigma,\tau}\onen := \xy
 (-8,14);(-8,-14); **[grey][|(4)]\dir{-} ?(.85)*[grey][|(3)]\dir{<}?(.5)*{\bigb{\mu}};
 (8,-14);(8,14); **[grey][|(4)]\dir{-} ?(.85)*[grey][|(3)]\dir{<}?(.5)*{\bigb{\lambda}};
 (8,10)*{};(-8,10)*{} **[grey][|(4)]\crv{(8,2) & (-8,2)} ?(.25)*[grey][|(3)]\dir{<}?(.5)*{\bigb{\nu}};
 (-8,-10)*{};(8,-10)*{} **[grey][|(4)]\crv{(-8,-2) & (8,-2)} ?(.25)*[grey][|(3)]\dir{<}?(.5)*{\bigb{\sigma}};
 (-13,14)*{\scriptstyle b+i-j};
 (13,14)*{\scriptstyle a+i-j};
 (-10,-14)*{\scriptstyle b};
 (10,-14)*{\scriptstyle a};
 (5,-7)*{\scriptstyle j};
 (-5,7)*{\scriptstyle i};
 (18,2)*{b^+(s_{\tau})};
 (16,-8)*{n};
 (24,0)*{};
\endxy,\quad\quad g^{a,b,i,j}_{\lambda,\mu,\nu,\sigma,\tau}\onen  := \xy
 (-8,14);(-8,-14); **[grey][|(4)]\dir{-} ?(.9)*[grey][|(3)]\dir{>}?(.5)*{\bigb{\lambda}};
 (8,-14);(8,14); **[grey][|(4)]\dir{-} ?(.9)*[grey][|(3)]\dir{>}?(.5)*{\bigb{\mu}};
 (8,10)*{};(-8,10)*{} **[grey][|(4)]\crv{(8,2) & (-8,2)} ?(.3)*[grey][|(3)]\dir{>}?(.5)*{\bigb{\nu}};
 (-8,-10)*{};(8,-10)*{} **[grey][|(4)]\crv{(-8,-2) & (8,-2)} ?(.3)*[grey][|(3)]\dir{>}?(.5)*{\bigb{\sigma}};
 (-13,14)*{\scriptstyle a+i-j};
 (13,14)*{\scriptstyle b+i-j};
 (-10,-14)*{\scriptstyle a};
 (10,-14)*{\scriptstyle b};
 (5,-7)*{\scriptstyle j};
 (-5,7)*{\scriptstyle i};
 (18,2)*{b^+(s_{\tau})};
 (16,-8)*{n};
 (24,0)*{};
\endxy
\end{equation*}
where  $b^+\maps {\rm Sym}\longrightarrow \End(\onen)$  is an isomorphism from the ring of symmetric functions to $\End(\onen)$ defined on complete and elementary symmetric functions as
\begin{equation}\label{defbh}
   b^+(h_i)= \xy
(0,0)*{\rbub};
 (5,4)*{n};
 (6,0)*{};
 (5,-2)*{\scriptstyle i};
\endxy, \quad
b^+(\varepsilon_i)=(-1)^i\xy
(0,0)*{\lbub};
 (5,4)*{n};
 (6,0)*{};
 (5,-2)*{\scriptstyle i};
\endxy.
\end{equation}
For more details see \cite[Section 6.2]{BHLZ}.

 This basis can be used to show that the Hom categories $\Us(n,m)$ between objects $n,m \in \Ob(\Us)$ are graded fibred strongly object adapted  cellular categories over the graded ring $R={\rm Sym}$ of symmetric functions in infinitely many variables.   Given that every 1-morphism in $\Us$ is isomorphic to a direct sum of indecomposables, it suffices to establish the fibred SOACC axioms for maps between indecomposables.

The objects $n \in \Ob(\Us)$ are parameterized by the weight lattice of $\mathfrak{sl}_2$ which we identify with $\Z$.  For convenience, assume $n,m \geq 0$, then the SOACC structure on $\Us(n,m)$ is given by
\begin{itemize}
  \item $\Lambda = {}_m\mathbb{B}_{n}$,
  \item The order is that $(b,a) < (b+j,a+j)$ for $j \geq 0$, i.e. the order is given by the thickness of the strands.
  \item One has $A_{(b,a)} \cong {\rm Sym}_b \boxtimes {\rm Sym}_a \boxtimes R$, with a basis over $R$ given by Schur polynomials $\l \in P(a)$ and $\mu \in P(b)$.
  \item The set $E((b,a), (b-j,a-j))$ is parametrized by $\sigma \in P(j)$.
  \item If $\lambda = \mathcal{F}^{(b)} \mathcal{E}^{(a)}1_n$ and $X=\mathcal{F}^{(b-j)} \mathcal{E}^{(a-j)}1_n$ then
 \[
E(\Lambda,X) :=
\vcenter{\xy
 (-8,6);(-8,-14); **[grey][|(4)]\dir{-} ?(.85)*[grey][|(3)]\dir{<};
 (8,-14);(8,6); **[grey][|(4)]\dir{-} ?(.85)*[grey][|(3)]\dir{<};
 (-8,-10)*{};(8,-10)*{} **[grey][|(4)]\crv{(-8,-2) & (8,-2)} ?(.25)*[grey][|(3)]\dir{<}?(.5)*{\bigb{\sigma}};
 (-13,6)*{\scriptstyle b-j};
 (13,6)*{\scriptstyle a-j};
 (-10,-14)*{\scriptstyle b};
 (10,-14)*{\scriptstyle a};
 (5,-7)*{\scriptstyle j};
 (16,-1)*{n};
 (24,0)*{};
\endxy}
 \qquad \qquad
M(X,\Lambda) :=
\vcenter{\xy
 (-8,14);(-8,-6); **[grey][|(4)]\dir{-} ?(.85)*[grey][|(3)]\dir{<};
 (8,-6);(8,14); **[grey][|(4)]\dir{-} ?(.85)*[grey][|(3)]\dir{<};
 (8,10)*{};(-8,10)*{} **[grey][|(4)]\crv{(8,2) & (-8,2)} ?(.25)*[grey][|(3)]\dir{<}?(.5)*{\bigb{\nu}};
 (-10,14)*{\scriptstyle b};
 (10,14)*{\scriptstyle a};
 (-13,-6)*{\scriptstyle b-j};
 (13,-6)*{\scriptstyle a-j};
 (-5,7)*{\scriptstyle j};
 (16,4)*{n};
 (24,0)*{};
\endxy}
\]
\end{itemize}

Using thick calculus relations, arguing as in the proof of \cite[Lemma 5.2]{KLMS} using \cite[Lemma 4.16]{KLMS} one can establish axiom (3) for a fibred SOACC.
\end{example}

\subsection{Decategorification}
\label{subsec-decat}

First we treat the Grothendieck decategorification, and then the trace decategorification. Let $\dot{\CC}$ denote the Karoubi envelope of $\CC$.

In \cite{EWGR4SB} it was proven that the Karoubi envelope of $\DC$ has one indecomposable object for each $\l \in \L$. The proof used is actually quite general, and applies almost verbatim to a general graded SOACC, and in fact to graded fibred SOACCs as well. We repeat the proof here.

\begin{defn} Let $\Bbbk$ be a commutative complete local ring. A graded $\Bbbk$-algebra $A$ is \emph{nice} if it is non-negatively graded, $A^0 \cong \Bbbk$, and it is finite-dimensional in each degree. A graded fibred SOACC is
\emph{nice} if $R$ and each $A_\l$ are nice graded $\Bbbk$-algebras.
\end{defn}

%

Let $\CC$ be a nice graded $\Bbbk$-linear fibred SOACC and let $\CC^{\ge \l}$ denote the quotient of $\CC$ by the ideal $\CC_{ \ngtr \l}$.  In $\CC^{\ge \l}$ the object $\lambda$ is a minimal cell.  For each object $X$ we can form the homogenous ideal
\[
 J_{\lambda}(X) := \left\{
 c_{S,a,T}^{\l} \mid \deg(a) > 0 , S\in M(\lambda,X), T\in E(X,\lambda)
 \right\}.
\]

\begin{lemma} \label{lem_Jideal}
For each object $X$ in $\CC^{\ge \l}$ and $k\leq 0$ the homogeneous component $(J_{\lambda}(X))^k$ of ideal $J_{\lambda}(X)$ is contained in the graded Jacobson radical of the ring $\End_{\CC^{\ge \l}}(X)$.
\end{lemma}

\begin{proof}
It suffices to show that any element $y \in (J_{\l}(X))^k$ for $k \leq0$ is nilpotent. Such an element $y$ is a finite sum of $c_{S,a,T}^{\l}$ with
$\deg(a)>0$, so that $\deg(S) + \deg(T) < 0$.
 We know that
\[
c_{S,a,T} c_{U,a',V} = c_{S,a'',V}
 \]
where $a'' = a c_T c_U a'$. So the set of all $S \in M(\l,X)$ (resp. $T\in E(X,\l)$) with nonzero coefficients in $y$ does not grow when taking powers of $y$.

In order for $c_T c_U \ne 0$ one must have $deg(T) + deg(U) \ge 0$. Thus if $c_{S,a,T} c_{U,a',V}$ is nonzero, then
$\deg(S) < \deg(U)$ since $\deg(T) + \deg (S) < 0$ and $\deg(T) + \deg(U) \ge 0$.  Similarly,
$\deg(V) < \deg(T)$.  Thus,
$\deg(S) + \deg(V) < \deg(S) + \deg(T)$  and $\deg(S) + deg(V) < \deg(U) + \deg(V)$.
In particular, for any surviving term in $y^2$ of the form $c_{S,a,V}$, the sum $\deg(S) + \deg(V)$ is
strictly smaller than the corresponding sum for some term in $y$.
But the maximum of $\deg(S)$ plus the maximum of $\deg(T)$ is bounded above, and the minimum of $\deg(S)$ plus the minimum of $\deg(T)$ is bounded below. These same bounds apply to all powers of $y$. Thus all $c_{S,a,V}$ must vanish for some power of $y$.
\end{proof}

\begin{prop} \label{prop:grothofcellcat} Let $\Bbbk$ be a commutative complete local ring, and let $\CC$ be a nice graded $\Bbbk$-linear fibred SOACC. Then $\dot{\CC}$ has one indecomposable
isomorphism class for each $\l \in \L$ and each shift $k \in \ZM$, and its Grothendieck group is isomorphic to $\ZM[q,q^{-1}] \cdot \L$. \end{prop}

\begin{proof}
Let $\l \in \L$. Write the identity of $\l$ as a sum of mutually orthogonal indecomposable idempotents:
\[
\1_\l = e_1 + \dots + e_n.
\]
Then the $e_i$ remain orthogonal idempotents in the quotient of $\End^0(\l)$ by $\CC_{< \l}$, which is $A^0_\l \cong \Bbbk$.  However, $\Bbbk$ has only one non-zero idempotent, the identity. Therefore, there is a unique idempotent (say $e_1$) which is non-zero in this quotient. We define $B_\l$ to be the image of this idempotent. It remains to show that any indecomposable object in $\dot{\CC}$ is isomorphic to some shift of some $B_\l$.


Let $B$ be an arbitrary indecomposable object in $\dot{\CC}$. That is, $B$ can be viewed as an object $X \in \Ob(\CC)$ paired with an indecomposable idempotent $e \in \End^0(X)$. For any $\l
\in \L$ the ring $\End_{\CC^{\ge \l}}(X)$ is a quotient of $\End(X)$. Let $\l$ be maximal amongst the set of elements for which $e \in \End_{\CC^{\ge \l}}(X)$ is nonzero. Equivalently, if we write $e$ in the cellular basis
\[
e = \sum c^\mu_{S,a,T}
\]
then $\l$ is maximal such that some term with $a \ne 0$ appears. Hence in $\End_{\CC^{\ge \l}}(X)$ we can write
\[
e = \sum c^\l_{S,a_{S,T},T}
\]
where the sum is over $S \in M(\l,X)$ and $T \in E(X,\l)$, and $a_{S,T}$ is some homogeneous element of $A_\l$. Note that the degree of $S$, the degree of $T$, and the degree of $a_{S,T}$ sum to zero.

For any $U \in M(\l,X)$ and $V \in E(X,\l)$, one can consider $c_V \circ e \circ c_U \in \End_{\CC^{\ge \l}}(\l) = A_\l$. Suppose that this element of $A_\l$ is in the maximal graded ideal
for each such $U$ and $V$ (as it must be whenever the degree of $U$ plus the degree of $V$ is nonzero). Then, by expanding $e^3 = e$ we conclude that $a_{S,T} \in A_\l$ is in the maximal
graded ideal for all $S,T$. However, by Lemma ~\ref{lem_Jideal} this would imply that $e$ is in the graded Jacobson radical of the ring $\End_{\CC^{\ge \l}}(X)$. This is a contradiction, as no non-zero idempotent lives in the Jacobson radical.

So there exists some $U \in M^k(\l,X)$ and $V \in E^{-k}(X,\l)$ with $c_V \circ e \circ c_U$ invertible in $\End^0_{\CC^{\ge \l}}(\l) \cong \Bbbk$. Letting $e_1$ denote the projection from
$\l$ to $B_\l$ as above, we see that $e_1 c_V e c_U e_1$ is still invertible in $\End^0_{\CC^{\ge \l}}(\l)$. The degree $-k$ morphism $p=e_1 c_V e$ from $X$ to $\l$ (resp. the degree $+k$
morphism $i=e c_U e_1$ from $\l$ to $X$) in $\CC$ induces a morphism $B \to B_\l$ (resp. $B_\l \to B$) in $\dot{\CC}$. The composition $pi$ is invertible in the quotient of $\End(B_\l)$ by
$\CC_{< \l}$, so it can not lie in the maximal ideal of $\End(B_\l)$. Therefore $pi$ is invertible in $\End(B_\l)$. In particular, $ip$ is, up to invertible scalar, an idempotent in the
local ring $\End(B)$, so it too must be invertible. Therefore, $i$ and $p$ give isomorphisms between $B$ and $B_\l(-k)$. \end{proof}

Now we compute the trace decategorification under similar assumptions.

\begin{prop} \label{prop:traceofcellcat} Let $\Bbbk$ be a commutative complete local ring with perfect quotient field. Let $\CC$ be a nice graded $\Bbbk$-linear fibred SOACC. Let $\dot{\CC}$
denote its Karoubi envelope, and let $\CC^*$ denote the graded category with translation obtained from $\CC$, as in \cite[Section 2.4]{BHLW}. Then the canonical maps induce isomorphisms of
$\ZM[q,q^{-1}]$-modules \begin{equation} \label{traceofCdot} [\dot{\CC}] \ot_{\ZM} \Bbbk \simto \Tr(\CC), \end{equation} \begin{equation} \label{traceofCstar} \bigoplus_{\l \in \L} A_\l \simto \Tr(\CC^*). \end{equation}
Thus for an SOACC one has
\begin{equation} \label{traceofCstarSOACC} [\dot{\CC}] \ot_{\ZM[q,q^{-1}]}
R \simto \Tr(\CC^*), \end{equation}
 where in \eqref{traceofCstarSOACC} the specialization of $\ZM[q,q^{-1}]$ factors through $q=1$. \end{prop}

\begin{proof}
Let $H:=\bigoplus_{X \in \Ob(\mathcal{C}^{\ast})}\Hom(X,X)$ and consider the subgroup $K:= \bigoplus_{\lambda \in \Lambda} A_\l \langle 1_{\lambda} \rangle$.  We will use \cite[Proposition 3.4]{BHLZ} to deduce that $\Tr(\mathcal{C}^{\ast}) \cong K$.
To apply this proposition we must define a map $\pi \maps H \to K$ such that
\begin{enumerate}[(A)]
  \item $\pi(f) = f$ for all $f \in K$,
  \item there is an equality of trace classes $[\pi(f)]=[f]$ for all $f\in H$, and
  \item $\pi(gh)=\pi(hg)$ for every $g\in \Hom(X,Y)$ and $h \in \Hom(Y,X)$.
\end{enumerate}

We define a homomorphism $p \maps H \to H$.  Using the first axiom of a fibred SOACC we can write any $f\in \Hom(X,X)$ as
\[
 f = \sum_{\lambda, S,T} c_{S,a^\l_{S,T},T}^{\lambda}, \qquad a_{S,T}^{\lambda} \in A_\l
\]
where the sum is over $\lambda \in \Lambda$, $T \in E(X,\lambda)$, $S \in M(\lambda, X)$.  Recall that $c_{S,a,T}^{\lambda} := c_S \circ a \circ c_T$.  Define
\[
p(f) := \sum_{\lambda, S, T} a_{S,T}^{\lambda}\;  c_{T}\circ c_{S} = \sum_{\l, S, T} a_{S,T}^\l \phi^\l_X(S,T).
\]
It is clear that there is an equality of trace classes $[p(f)]=[f]$.  Using axiom (4) of an fibred SOACC it follows that  $p(f)=f$ for any $f\in K$.

Using the third axiom and the descending chain condition, it follows that for every $f\in H$ there exist a $k \geq 0$ such that $p^k(f)\in K$.  Let $\pi \maps H\to K$ be defined by $\pi(f) = p^k(f)$ where $k$ is chosen as above.  Since it is true for $p$, we have $\pi(f)=f$ for $f\in K$, and $[\pi(f)]=[f]$.  Finally, condition (C) above follows by induction over the partial order using axiom 3 of a fibred SOACC to deduce that $p(gh)=p(hg)$ modulo lower terms. For $g$ and $h$ both in a minimal cell $\l$, it follows by Lemma \ref{pairmultSOAaCC}.

A similar argument working in $\CC$ shows establishes the isomorphism \eqref{traceofCdot}, since the degree zero component of $A_\l$ is just $\Bbbk$.

For sake of completeness, we spell out the inductive proof of condition (C); the reader who believes it may skip to the next section. We wish to show that $p(gh) = p(hg)$ for all $g \in
\Hom(X,Y)$ and $h \in \Hom(Y,X)$. Clearly it suffices to let $g = c^\l_{S,a,T}$ and $h = c^\mu_{U,b,V}$ for some $\l,\mu,S,T,U,V,a,b$. As mentioned above, the case where $\l = \mu$ is
minimal follows from Lemma \ref{pairmultSOAaCC}. We now assume that $p(gh) = p(hg)$ whenever $g \in \CC_{< \l}$ and $h \in \CC_{\le \mu}$, or $g \in \CC_{\le \l}$ and $h \in \CC_{< \mu}$. We assume, without loss of generality, that $\l$ is not greater than $\mu$.

First suppose that $\l \ne \mu$. One has \[ gh = c_S a c_T c_U b c_V,\] expressing $gh$ as a composition of six morphisms. Note that $c_T c_U b c_V$ is in both ideals $\CC_{\le \l}$ and
$\CC_{\le \mu}$, so that it lies in the ideal $\CC_{< \mu}$. Hence $p(gh) = p(c_T c_U b c_V c_S a)$ by induction. However, $c_U b c_V c_S a$ is also in both ideals, so $p(c_T c_U b c_V c_S
a) = p(c_U b c_V c_S a c_T) = p(hg)$ by induction. Therefore $p(gh)=p(hg)$.

Now suppose that $\l = \mu$. One has
\[ c_T c_U = \phi^\l_X(T,U) + x \]
where $x \in \End(\l)$ is in $\CC_{< \l}$. Meanwhile,
\[ c_V c_S = \phi^\l_Y(V,S) + y \]
where $y \in \CC_{< \l}$. Thus
\[ p(gh) = p(c_S a \phi(T,U) b c_V + c_S a x b c_V) = p(a \phi(T,U) b c_V c_S + a x b c_V c_S). \]
The equality of first terms comes from the definition of $p$, while the equality of second terms comes from induction, using the fact that $a x b c_V$ is in $\CC_{< \l}$. Continuing the computation, we have
\[ p(gh) = p(a b \phi(T,U) \phi(V,S) + a b \phi(T,U) y + a b x \phi(V,S) + a b x y).\]
An identical computation shows that $p(hg)$ is equal to the same quantity.
\end{proof}

%

\section{Trace decategorification of $\DC$}
\label{sec-traceofD}

In Proposition \ref{prop:traceofcellcat} we have computed the underlying $\Bbbk$-module of $\Tr(\DC)$ and the underlying $R$-module of $\Tr(\DC^*)$. It remains to compute their monoidal
structures.

Because $\Tr(\DC)$ is generated by identity maps, the ring structure on $\Tr(\DC)$ will agree with that of $[\dot{\DC}]$ under the isomorphism of \eqref{traceofCdot}. Therefore, we have:

\begin{thm} \label{mainthm1} As a ring, $\Tr(\DC) \cong \HB \ot_{\ZM} \Bbbk$. Under this isomorphism, the identity $1_{B_s}$ of the Bott-Samelson bimodule $B_s$ is sent to the Kazhdan-Lusztig generator $b_s
= v(1+T_s)$, for each simple reflection $s$. \end{thm}

Similarly, $\Tr(\DC^*)$ is generated by identity maps, together with $\End(\1) = R$, the endomorphism ring of the monoidal identity, also known as the center $Z(\DC)$. The ring structure on
the subring generated by identity maps agrees with the ring structure on $[\dot{\DC}]$ specialized at $q=1$; this is $\HB / (q-1) = \ZM[W]$, the group algebra of the Coxeter group $W$. To
avoid confusion, we denote the standard basis of the group algebra by $\{T_w\}_{w \in W}$, which agrees with the specialization of the standard basis $\{T_w\}$ inside the Hecke algebra. Thus
$1_{B_s}$ is sent to $(1+T_s)$. Meanwhile, the ring structure on $R \subset \Tr(\DC^*)$ agrees with the usual ring structure on $\End(\1)$. It remains to compute the cross relations between
$W$ and $R$, a computation which will only involve relations (5.1) and (5.2) from \cite{EWGR4SB}.

Let us place (5.2) on an annulus to compute its implication inside $\Tr(\DC^*)$.

\begin{equation}
{
\labellist
\small\hair 2pt
 \pinlabel {$f$} [ ] at 14 62
 \pinlabel {$sf$} [ ] at 204 62
 \pinlabel {$\pa_s(f)$} [ ] at 352 62
 \pinlabel {$=$} [ ] at 145 62
 \pinlabel {$+$} [ ] at 309 62
\endlabellist
\centering
\ig{.8}{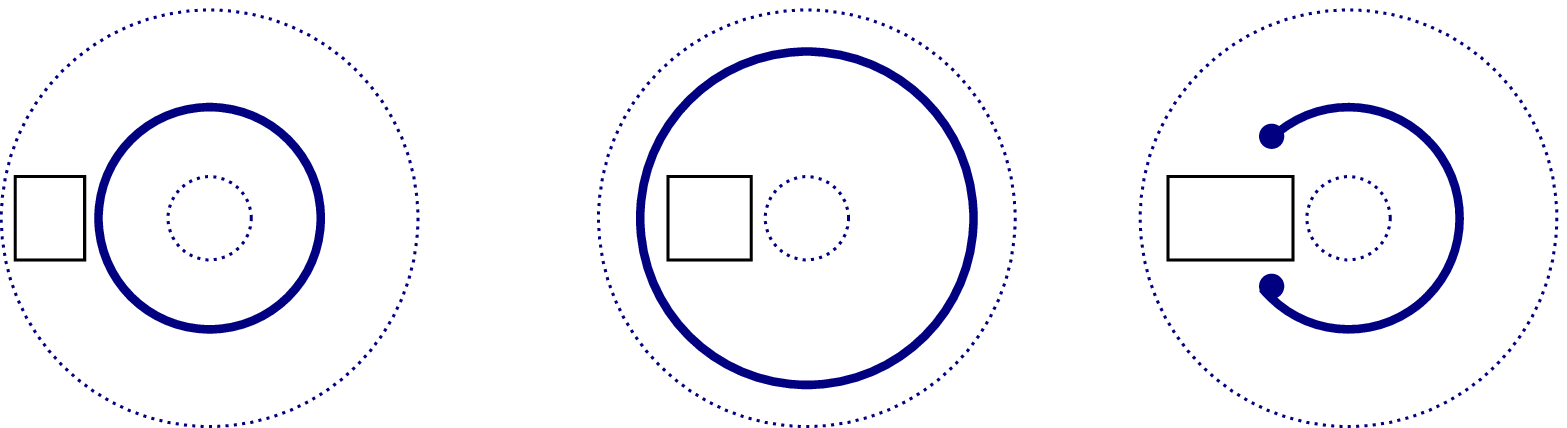}
}
\end{equation}

Using (5.1), the final diagram is equal to the polynomial $\pa_s(f) \a_s$, which by definition of divided difference operators is equal to $f - s\cdot f$. Therefore we obtain the relation
\begin{equation} \label{semidirectusingB} f 1_{B_s} = 1_{B_s} s \cdot f + (f - s \cdot f). \end{equation}
Rewriting this using $1_{B_s} = (1 + T_s)$, and subtracting $f$ from both sides, one has
\begin{equation} \label{semidirectusingT} f T_s = T_s s \cdot f. \end{equation}
This is exactly the relation in the semi-direct product $\ZM[W] \ltimes R$.

\begin{thm} \label{mainthm2} As a ring, $\Tr(\DC^*) \cong \ZM[W] \ltimes R$. \end{thm}

\begin{proof} The isomorphism of underlying vector spaces followed from Proposition \ref{prop:traceofcellcat}. The compatibility with the ring structure follows from the discussion above.
\end{proof}

Finally, we consider the canonical action of the trace of a category on its center. Given an element of the trace, viewed as a diagram on the annulus (with empty boundary), and an element of
the center, viewed as a diagram on the circle (with empty boundary), the action is to wrap the annulus around the circle to obtain another circle. However, because
\begin{equation} \ig{.5}{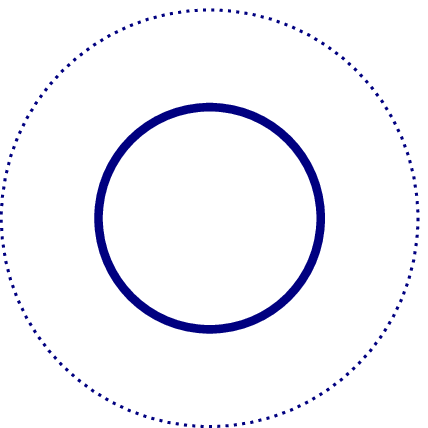} \quad = \quad 0, \end{equation}
one sees that the action of $1_{B_s}$ sends $1 \in Z(\DC)$ or $1 \in Z(\DC^*)$ to zero. The representation of $\HB$ (resp. its quotient $\ZM[W]$) where the Kazhdan-Lusztig generator $[B_s] =
v(1+T_s)$ is sent to zero is called the \emph{sign representation}, on which $T_s$ acts by $-1$.

\begin{thm} \label{mainthm3} The action of $\Tr(\DC)$ on $Z(\DC)$ is the sign representation of $\HB \ot_{\ZM} \Bbbk$. The action of $\Tr(\DC^*)$ on $Z(\DC^*)$ is the induction of the sign representation of
$\ZM[W]$ to $\ZM[W] \ltimes R$. \end{thm}

\bibliographystyle{plain}
\bibliography{GR4allSB}

\end{document}